\documentclass[reqno, 11pt]{amsart}
\usepackage[margin=1in]{geometry}
\usepackage{hyperref, amsrefs, graphicx, xcolor, amssymb}

\numberwithin{equation}{section}
\newtheorem{theorem}{Theorem}[section]
\newtheorem{corollary}[theorem]{Corollary}
\newtheorem{lemma}[theorem]{Lemma}
\newtheorem{proposition}[theorem]{Proposition}
\newtheorem{remark}[theorem]{Remark}

\newtheorem{definition}[theorem]{Definition}

\newcommand{\Per}{\operatorname{Per}}
\newcommand{\R}{{\mathbb{R}}}

\newcommand{\N}{{\mathbb{N}}}
\newcommand{\cI}{{\mathcal I}}
\newcommand{\cD}{{\mathcal D}}
\newcommand{\cS}{{\mathcal S}}
\newcommand{\eps}{{\varepsilon}}
\allowdisplaybreaks

\begin{document}
\title[Nonlocal minimal surfaces are generically unique and smooth]{Nonlocal minimal surfaces are generically unique, \\
and smooth in one extra dimension}

\author[S.~Dipierro]{Serena Dipierro}  \address{Serena Dipierro\\  Department of Mathematics and Statistics\\
	University of Western Australia,\\
	35 Stirling Highway, WA 6009 Crawley, Australia}
	\email{\href{mailto:serena.dipierro@uwa.edu.au}{serena.dipierro@uwa.edu.au}}
		
	\author[X.~Fern\'{a}ndez-Real]{Xavier Fern\'{a}ndez-Real}  \address{Xavier Fern\'{a}ndez-Real\\
School of Basic Sciences\\	
	École Polytechnique Fédérale de Lausanne\\
Station 8, 1015 Lausanne, Switzerland}
\email{\href{mailto:xavier.fernandez-real@epfl.ch}{xavier.fernandez-real@epfl.ch}}

	\author[E.~Valdinoci]{Enrico Valdinoci}  \address{Enrico Valdinoci\\Department of Mathematics and Statistics\\
	University of Western Australia,\\
	35 Stirling Highway, WA 6009 Crawley, Australia}
\email{\href{mailto:enrico.valdinoci@uwa.edu.au}{enrico.valdinoci@uwa.edu.au}}

\begin{abstract}
We study the fractional Plateau problem for $s$-minimal sets with prescribed
exterior datum. For fixed exterior data, minimizers need not be unique, and
singularities may occur beyond the critical dimension. We first prove a
generic uniqueness theorem: along any strictly increasing family of exterior
data, nonuniqueness occurs for at most countably many parameters. We then
show that one can make arbitrarily small perturbations for which
the interior regularity theory improves by one dimension. 
\end{abstract}


\maketitle 

\section{Introduction}
Nonlocal minimal surfaces arise as the natural generalization of classical minimal surfaces when long-range interactions are taken into account. While classical minimal surfaces are local minimizers of the classical perimeter functional, nonlocal minimal surfaces minimize a nonlocal version of perimeter, in which the interaction between points extends over the entire space.

This framework has been introduced and extensively studied
by Caffarelli, Roquejoffre,  and Savin
in~\cite{MR2675483}.
For disjoint (measurable) sets $E$, $F \subseteq \R^n$ and a fractional parameter $s \in (0,1)$, the $s$-interaction between~$E$ and~$F$ is defined as
$$ I_s(E,F):=\iint_{E\times F}\frac{dx\, dy}{|x-y|^{n+s}}.$$
Given a reference domain~$\Omega$, the $s$-perimeter of~$E$ in~$\Omega$ is
$$\Per_s(E,\Omega):=I_s(E\cap\Omega,E^c\cap\Omega)
+I_s(E\cap\Omega,E^c\cap\Omega^c)+
I_s(E\cap\Omega^c,E^c\cap\Omega),$$
where the superscript~``$c$'' denotes the complement set in~$\R^n$.

A set~$E$ is called $s$-minimal in~$\Omega$ (and~$\partial E$ is called an $s$-minimal surface in~$\Omega$) if,
for every set~$F$ such that~$F\setminus\Omega$=$E\setminus\Omega$, we have that~$\Per_s(E,\Omega)\le\Per_s(F,\Omega)$.

As $s \nearrow1$, the fractional perimeter converges (after suitable normalization) to the classical perimeter, and nonlocal minimal surfaces converge to classical minimal surfaces (see~\cite{MR3586796, MR1942130, MR2782803, MR2765717}).
Nonlocal minimal surfaces are also related to front propagation (see~\cite{MR2564467}) and long-range phase coexistence models (see~\cite{MR2948285}). Despite several similarities to their classical counterparts, nonlocal minimal surfaces often exhibit peculiar behaviors: for example, they tend to stick at the boundary of the domain, thereby producing boundary jump discontinuities
and the sudden divergence of the
boundary derivative (see~\cite{MR3596708}).

The goal of this paper is to establish the genericity of the uniqueness and smoothness of minimizers with respect to the external data. For nonlocal minimal surfaces, this is far from obvious and certainly cannot be taken for granted: in fact, the only generic result currently available for nonlocal minimal surfaces is that of~\cite{MR4104542}, which shows that stickiness and boundary discontinuity are generic and can be produced by arbitrarily small perturbations of the external data,
thus suggesting a ``robust behavior'' of the specific ``pathologies'' created by long-range interactions.
It is therefore particularly interesting that, in contrast, properties related to well-posedness and regularity, such as uniqueness and interior smoothness, also turn out to be generic.

\subsection{Generic uniqueness of nonlocal minimal surfaces}

We recall that multiplicity for nonlocal minimal sets can occur,
namely there exist simple cases of domains and external data
for which there are at least two different $s$-minimal sets,
see e.g.~\cite[Theorem~1.6]{MR4184583}
for an explicit example. We show, however, that uniqueness of minimizers is a generic property and can be recovered up to arbitrarily small 
modifications of the external data. More precisely, we have:


\begin{theorem}\label{gene-uni}
Let $\Omega\subset \R^n$ be a bounded domain. Let~$I\subseteq\R$ and, for all~$t\in I$, consider a family of sets~$G^{(t)}\subseteq\R^n$ and assume that, for each~$t\in I$, there is an admissible set of finite
$s$-perimeter with exterior datum~$G^{(t)}$.  Assume, moreover, that for all~$t_1$, $t_2\in I$ with~$t_1<t_2$, it holds that~$ G^{(t_1)}\setminus\Omega\subsetneq
G^{(t_2)}\setminus\Omega$.

Then, there exists a countable set $N\subset I$ such that, for all $t\in I\setminus N$, there is a unique $s$-perimeter minimizer in $\Omega$ with exterior datum $G^{(t)}$. 
\end{theorem}

We remark that, throughout the paper, all identities between sets (equalities, inclusions, and strict inclusions) are assumed to hold up to sets of measure zero.

As it will be clear from its proof, Theorem~\ref{gene-uni}
actually works for a very wide class of anisotropic nonlocal minimal surfaces as well, since it only relies
on the weak maximum principle (in particular, general nonnegative and symmetric interaction kernels can be taken into account; see e.g.~\cite{MR3981295} and the references therein for nonlocal minimal surfaces in
anisotropic settings). We refer to Proposition~\ref{prop:general-kernel-generic-uniqueness} for a more precise statement in this direction.

\subsection{Generic regularity of nonlocal minimal surfaces}

We now present our formulation of the generic regularity statement proved in this paper.  The statement is modeled on the classical generic regularity theorems for the Plateau problem, starting from the work of Hardt--Simon in the first singular dimension~\cite{MR809969}. In the nonlocal Plateau problem, however, the prescribed object is not a codimension-two boundary, but rather the exterior datum outside the domain. Accordingly, the perturbations that we will consider below are small modifications of this exterior datum.

The critical dimension for nonlocal minimal cones is
\begin{equation}\label{DIMECR}
n_s^*:=\max\big\{m\in\N\text{ such that every $s$-minimal cone in $\R^m$ is a halfspace}\big\}.
\end{equation}
Due to improvement of flatness results, the dimension~$n_s^*$
is the one ensuring interior regularity of~$s$-minimal surfaces. The exact value of~$n_s^*$ is not known in general. We recall that $n_s^*\ge2$ for all $s\in(0,1)$ and that $n_s^*\ge7$ for $s$ sufficiently close to~$1$; see~\cite{MR3090533, MR3107529}.

 The point of the theorem below is that this estimate improves by one dimension along a generic perturbation of the exterior datum, and gives complete regularity in the first dimension above the critical one.

\begin{theorem}[Generic regularity for $n = n_s^*+1$]\label{thm:main1}
Let $s\in(0,1)$, let $n = n_s^*+1$, and let $G\subseteq\R^n\setminus B_1$ be an exterior datum such that $\partial (  G\cup B_1)$ has locally finite $(n-1)$-Minkowski content. 

Then, there are $L^1_{\rm loc}$ small perturbations $G'$ of $G$ (see Definition~\ref{def:perturbations}) with the property that there exists a smooth $s$-perimeter minimizer $E'$ in $B_1$ with $E' = G'$ in $\R^n\setminus B_1$. 

Moreover, $E'$ is the unique minimizer.  
\end{theorem}

The terminology used in the statement of Theorem~\ref{thm:main1} can be made
precise, according to the following definitions.

\begin{definition}[Perturbations]
\label{def:perturbations}
    Let $\Omega\subset \R^n$ be a bounded domain. Given $\eps > 0$ and $G\subset \R^n\setminus \Omega$, we say that $G_\eps\subset \R^n\setminus \Omega$ is a perturbation of $G$ in $L^1_{\rm loc}$ of size $\eps$  if it holds 
    \[
    \int_{\R^n} |\chi_G - \chi_{G_\eps}|d\gamma \le \eps, 
    \]
    where $\chi_A$ is the characteristic function of a set $A\subset \R^n$, and $\gamma$ is a given Gaussian measure\footnote{One may also choose other metrics characterizing small $L^1_{\rm loc}$ perturbations.
    We chose this setting since it allows a clear and simple statement.    }. We say that $\{G_{\eps_k}\}_{k\in \N}$ is a family of small perturbations if each of them is a perturbation of size $\eps_k$ with $\eps_k\searrow 0$.
\end{definition}
\begin{definition}[Minkowski content]\label{def:minkowski_content}
    We recall that a set $E\subset \R^n$ has locally finite $(n-1)$-dimensional upper Minkowski content if, for any $R > 1$, 
    \[
    \limsup_{\delta\searrow0}\frac{|E_\delta \cap B_R |}{\delta}<+\infty,
\qquad{\mbox{where}}\qquad E_\delta :=\bigcup_{p\in E}
B_\delta (p).
\]
\end{definition}

In particular, given the above settings, Theorem~\ref{thm:main1}
ensures that
there are generic perturbations of the exterior datum yielding smooth minimizers
(which are also unique). 

More generally, to discuss interior regularity, it is customary to divide the boundary of a
set~$E$ into two portions, namely~${\rm Reg}(\partial E)$,
consisting of points where~$\partial E$ is locally a hypersurface of class~$C^\infty$
(or equivalently of Lipschitz class, see~\cite[Theorem~1.1] {MR3680376})
and~${\rm Sing}(\partial E):=\partial E\setminus{\rm Reg}(\partial E)$. Then,
Theorem~\ref{thm:main1}
says that, up to small perturbations of the exterior datum, the boundary of the minimizer is formed exclusively of regular points in dimension $n = n_s^*+1$. 

For higher dimensions, the general dimension-reduction theory for nonlocal minimal surfaces gives singular sets of Hausdorff dimension at most $n-n_s^*-1$; see~\cite[Theorems~10.3 and~10.4]{MR2675483} 
(or~\cite[page~871]{MR3107529}). By allowing arbitrarily small perturbations, we can improve this size by one full dimension as well: 

\begin{theorem}[Generic regularity for $n \ge n_s^*+2$]\label{thm:main2}
Let $s\in(0,1)$, let $n \ge n_s^*+2$, and let $G\subseteq\R^n\setminus B_1$ be an exterior datum such that $\partial ( G\cup B_1)$ has locally finite $(n-1)$-Minkowski content. 

Then, there are $L^1_{\rm loc}$ small perturbations $G'$ of $G$ (see Definition~\ref{def:perturbations}) with the property that there exists an $s$-perimeter minimizer $E'$ in $B_1$ with $E' = G'$ in $\R^n\setminus B_1$ such that 
\[
\dim_{\mathcal H}\big({\rm Sing}(\partial E')\cap B_1\big)\le n-n_s^*-2.
\]
Moreover, $E'$ is the unique minimizer.  
\end{theorem}

 We refer the reader to Theorem~\ref{GENER:REGULA0} below for the more precise construction and result yielding such an estimate.  {F}rom a technical perspective, the perturbation is constructed by combining two simple operations simultaneously: adding a small ball to the exterior datum and performing a homothety that is close to the identity.

Throughout, ``generic'' is meant in the perturbative sense: given any admissible exterior datum $G$ and any $L_{\rm loc}^1$-neighborhood of $G$, there exists $G'$ in that neighborhood such that every minimizer with exterior datum $G'$ satisfies the stated conclusion. The proof obtains this by constructing a monotone one-parameter family $G^{(t)}\to G$ and proving the conclusion for all $t$ outside a countable/null exceptional set.

Quite interestingly, comparing the results presented here with those in \cite[Theorem~1.1]{MR4104542}, one can say that, while interior regularity for nonlocal minimal surfaces is, in this sense, generic, {\em boundary regularity is not}.
\medskip 

 The result presented here may be viewed as a nonlocal counterpart of a classical theorem of Hardt and Simon (see~\cite{MR809969}), which states that, for generic choices of boundary data, the corresponding area-minimizing hypersurface in~$\R^8$ is smooth. This is particularly striking because complete regularity of area-minimizing hypersurfaces is guaranteed only up to ambient dimension~$7$, while singularities may occur starting in ambient dimension~$8$. See also~\cite{MR1243523}, where an analogous result is established for generic ambient metrics rather than boundary data. The result of Hardt--Simon has also been recently improved first in dimensions 9 and 10, and then up to 11, in recent works \cite{MR5082540, MR4753940, 2025arXiv250612852}.

The study of regularity properties in a generic setting for elliptic equations
also dates back to~\cite{MR387810} and has extended to several related
topics, including the classical obstacle problem (see~\cite{MR4179834, MR1967031}), the thin obstacle problem (see~\cite{MR4228865, MR4649881, colombocarducci}), 
and a large class of free boundary problems including those of
Alt-Caffarelli and Alt-Phillips (see~\cite{2023arXiv230813209F, fernandez2025continuity, FY25}). In particular, the methods and proofs we present here are adapted from the Bernoulli setting in \cite{2023arXiv230813209F}, which at the same time are inspired by the breakthroughs in \cite{MR4179834}.

In this spirit, the proof of the generic regularity results relies on
the construction of a monotone perturbation that produces a family of minimizers whose singular points form a space--time set \[
\cS:=\{(x,t):x\in{\rm Sing}(\partial E^{(t)})\cap B_1\},
\]
where the ``time'' $t$ here is the parameter describing the perturbation of the original exterior datum.
 
The separation estimate for the family implies that this set is, up to the exceptional nonuniqueness and discontinuity parameters, the graph of a Lipschitz time function.  Blow-ups at accumulation points of equal density then produce monotone $s$-minimal cones.  The cone-splitting argument rules out such configurations in the critical dimension and gives the sharp bound on the projection of~$\cS$ in higher dimensions.  A slicing argument in the parameter variable then yields the one-dimensional improvement for almost every exterior datum in the family.

\subsection{Organization of the paper}

The forthcoming Sections \ref{LAn2} and \ref{sec:generic-regularity} contain
the proofs, respectively, of Theorems~\ref{gene-uni} and of Theorems~\ref{thm:main1} and \ref{thm:main2}.

\section{Generic uniqueness: proof of Theorem~\ref{gene-uni}}\label{LAn2}

The goal of this section is to prove Theorem~\ref{gene-uni}.
To accomplish this goal, we need some ancillary results
related to the maximum principle for nonlocal minimal surfaces.

\subsection{Auxiliary results of maximum principle type}

We recall a simple, but useful, ``submodularity'' observation:

\begin{lemma}\label{LE1b} We have that
\begin{equation} \label{CONDc1}\Per_s(E\cap F,\Omega)+\Per_s(E\cup F,\Omega)\le
\Per_s(E,\Omega)+\Per_s(F,\Omega),\end{equation}
with equality holding if and only if, for a.e.~$(x,y)\in \R^{2n}\setminus(\Omega^c)^2$,
\begin{equation} \label{CONDc2}(\chi_E(x)-\chi_E(y))(\chi_F(x)-\chi_F(y))\ge0.\end{equation}
\end{lemma}

\begin{proof} Given~$u$, $v:\R^n\to\R$, we set~$m(x):=\min\{u(x),v(x)\}$ and~$M(x):=\max\{u(x),v(x)\}$ and we
have (see e.g. equations~(40) and~(41) in~\cite{MR3081641}) that, for all~$x$, $y\in\R^n$,
$$ |m(x)-m(y)|^2+|M(x)-M(y)|^2\le|u(x)-u(y)|^2+|v(x)-v(y)|^2,$$
with equality holding if and only if
$$ (u(x)-u(y))(v(x)-v(y))\ge0.$$
In particular, taking~$u:=\chi_E$ and~$v:=\chi_F$, we find that
\begin{equation} \label{CONDc3} |\chi_{E\cap F}(x)-\chi_{E\cap F}(y)|^2+|\chi_{E\cup F}(x)-\chi_{E\cup F}(y)|^2\le|\chi_E(x)-\chi_E(y)|^2+|\chi_F(x)-\chi_F(y)|^2,\end{equation}
with equality holding if and only if~\eqref{CONDc2} holds true for these~$x$ and~$y$.

Thus, integrating~\eqref{CONDc3} over~$(x,y)\in\R^{2n}\setminus(\Omega^c)^2$, one obtains the inequality in~\eqref{CONDc1}, as well as the fact that equality holds in~\eqref{CONDc1} if and only if~\eqref{CONDc2} holds true for a.e.~$(x,y)\in
\R^{2n}\setminus(\Omega^c)^2$.
\end{proof}

\begin{figure}[h]
\centering
\includegraphics[width=0.6\textwidth]{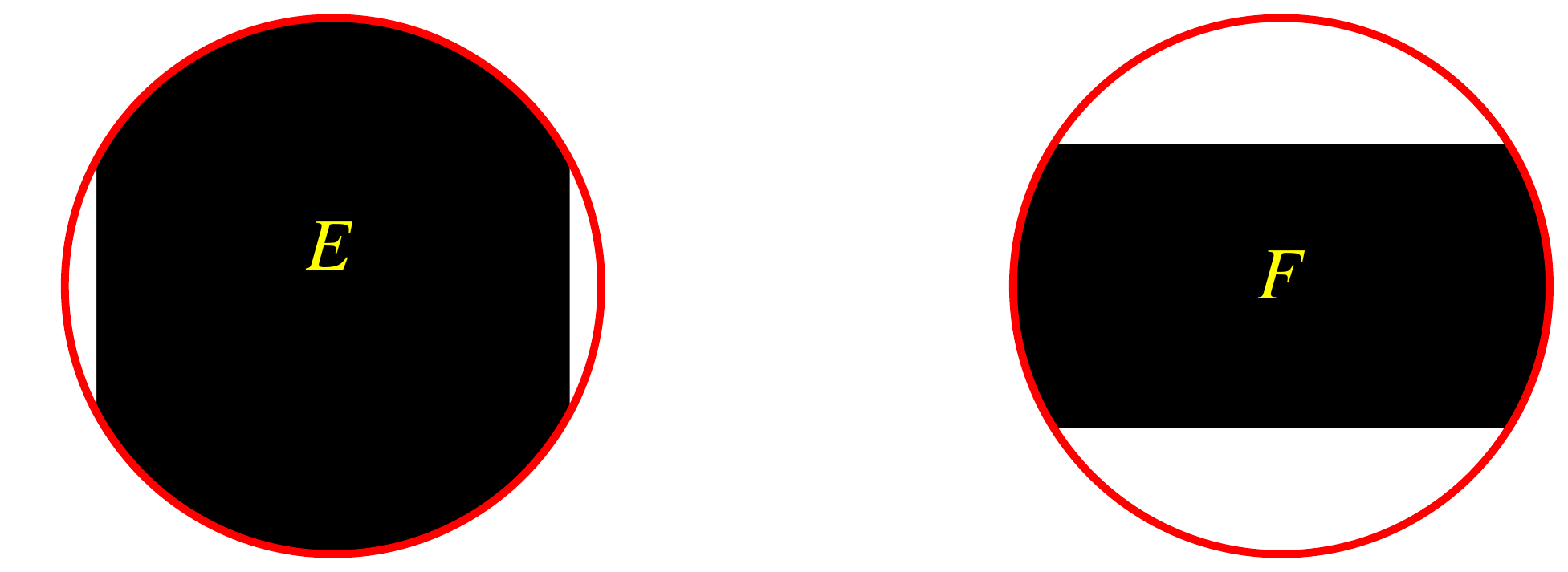}
\caption{\sl Examples of two sets~$E$ and~$F$ for which~$\Per(E\cap F,\Omega)+\Per(E\cup F,\Omega)=\Per(E,\Omega)+\Per(F,\Omega)$. Notice that
taking~$x\in E\setminus F$ and~$y\in F\setminus E$ we have that~$(\chi_E(x)-\chi_E(y))(\chi_F(x)-\chi_F(y))\le0$, in contrast with~\eqref{CONDc2},
showing that the analogue of Lemma~\ref{LE1b}
is false in the classical case.
}\label{FIGURA00}
\end{figure}

We stress that the result in Lemma~\ref{LE1b} is genuinely nonlocal, and the classical perimeter does not
satisfy properties of this type (not even for connected domains~$\Omega$, and not even in the plane, see
Figures~\ref{FIGURA00} and~\ref{FIGURA} for explicit counterexamples).

We deduce from Lemma~\ref{LE1b} the following weak maximum principle for nonlocal minimal surfaces. We recall that throughout the paper, all identities between sets (equalities, inclusions, and strict inclusions) are assumed to hold up to sets of measure zero:

\begin{corollary} \label{Cor1p}
Let~$E_0$, $ F_0\subseteq\R^n$ and suppose that~$E_0\setminus\Omega\subseteq F_0\setminus\Omega$.

Let~$E$ be $s$-minimal in~$\Omega$ with~$E\setminus\Omega=E_0\setminus\Omega$. 

Let~$F$ be $s$-minimal in~$\Omega$ with~$F\setminus\Omega=F_0\setminus\Omega$. 

Then, either
$$E\subseteq F$$
or
\begin{equation}\label{inc0}E_0\setminus\Omega= F_0\setminus\Omega\quad{\mbox{and}}\quad
F\subseteq E.\end{equation}
\end{corollary}

\begin{proof} By construction, $(E\cup F)\setminus\Omega=
(E_0\cup F_0)\setminus\Omega=F_0\setminus\Omega$ and therefore, by the minimality of~$F$, we have that~$\Per_s(F,\Omega)\le\Per_s(E\cup F,\Omega)$.

Similarly, $(E\cap F)\setminus\Omega=
(E_0\cap F_0)\setminus\Omega=E_0\setminus\Omega$ and therefore, by the minimality of~$E$, we have that~$\Per_s(E,\Omega)\le\Per_s(E\cap F,\Omega)$.

These observations and~\eqref{CONDc1} yield that~\eqref{CONDc2} holds true for a.e.~$(x,y)\in \R^{2n}\setminus(\Omega^c)^2$.

Suppose now that~$E\not\subseteq F$ and pick~$x\in\R^n$ such that~$x\in E\setminus F$ (up to removing a set of measure zero). {F}rom our assumptions on~$E_0$ and~$F_0$ it follows that~$x\in\Omega$.
Note that~$\chi_E(x)=1$ and~$\chi_F(x)=0$, which, combined with~\eqref{CONDc2} yields that, for a.e.~$y\in\R^n$,
$$ 0\ge-(\chi_E(x)-\chi_E(y))(\chi_F(x)-\chi_F(y))=
(1-\chi_E(y))\chi_F(y)=\chi_{E^c}(y)\chi_{F}(y),$$
that is
$$ \chi_{E^c}(y)\chi_{F}(y)=0.$$
Consequently, for a.e.~$y\in F$, we have that~$\chi_{E^c}(y)=0$,
giving that~$F\subseteq E$ (in the measure-theoretic sense), which leads us to~\eqref{inc0}.
\end{proof}

It is interesting to observe that, as a byproduct of Corollary~\ref{Cor1p},
minimizers of the fractional perimeter preserve the symmetries of the domain and boundary data, for example, for a rotation invariant domain with a rotation invariant external datum, the fractional perimeter minimizers are necessarily rotation invariant (and same with respect to translations, reflections, etc.): see~\cite[Lemma A.1]{MR3596708} for a general result in this sense. Note that this symmetry property is specific for nonlocal minimizers, since minimizers of the classical perimeter do not inherit the symmetries of the boundary and the boundary data, see Figure~\ref{FIGURA} for a counterexample
(in particular, this shows that not every minimizer of the classical perimeter can be obtained as a limit
of nonlocal minimizers).

\begin{figure}[h]
\centering
\includegraphics[width=0.6\textwidth]{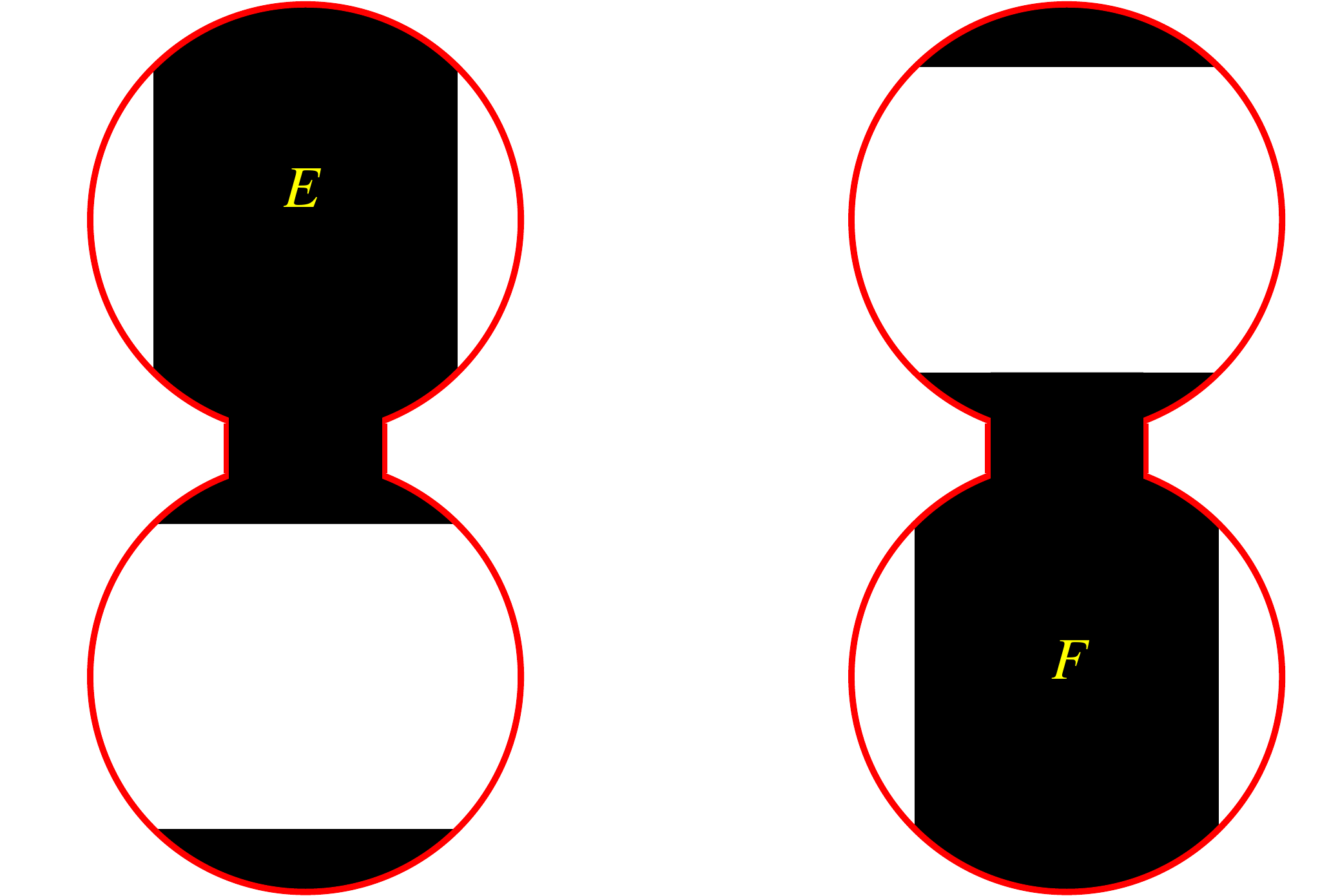}
\caption{\sl Examples of two minimizers~$E$ and~$F$ of the classical perimeter in a planar, connected domain.
Notice that neither of them is contained in the other, showing that the analogue of Corollary~\ref{Cor1p}
is false in the classical case. \\
Notice also that in this case~$\Per(E\cap F,\Omega)+\Per(E\cup F,\Omega)=\Per(E,\Omega)+\Per(F,\Omega)$,
but taking~$x\in E\setminus F$ and~$y\in F\setminus E$ we have that~$(\chi_E(x)-\chi_E(y))(\chi_F(x)-\chi_F(y))\le0$, in contrast with~\eqref{CONDc2},
showing that the analogue of Lemma~\ref{LE1b}
is false in the classical case.
}\label{FIGURA}
\end{figure}

{F}rom the weak maximum principle in Corollary~\ref{Cor1p}, we actually obtain an ``inclusion principle'', 
which is also conceptually interesting since it highlights that generic uniqueness
is implied by the property that one set is contained in the other which allows for weaker assumptions on the boundary data compared e.g. with the counterpart in the free boundary case~\cite{2023arXiv230813209F}. In our setting, the details on this topic go
as follows:

\begin{proposition}\label{ma:pro} Let~$I\subseteq\R$ and, for all~$t\in I$, consider a family of sets~$G^{(t)}\subseteq\R^n$
such that, for all~$t_1$, $t_2\in I$ with~$t_1<t_2$, it holds that~$ G^{(t_1)}\setminus\Omega\subsetneq
G^{(t_2)}\setminus\Omega$.

Assume that~$t_\star\in I$ is such that there exist at least two $s$-minimal sets in~$\Omega$
coinciding with~$ G^{(t_\star)}\setminus\Omega$ outside~$\Omega$.

Then, there exists a set~$S^{(t_\star)}\subseteq\Omega$ of positive Lebesgue measure such that
if~$I\ni t>t_\star$ and~$G^{(t)}_\diamond$ is an~$s$-minimal set in~$\Omega$ with~$G^{(t)}_\diamond\setminus\Omega=
G^{(t)}\setminus\Omega$, then
\begin{equation}\label{maxp1} S^{(t_\star)}\subseteq G^{(t)}_\diamond.\end{equation}

Furthermore, if~$I\ni \tau<t_\star$ and~$G^{(\tau)}_\diamond$ is an~$s$-minimal set in~$\Omega$ with~$G^{(\tau)}_\diamond\setminus\Omega=
G^{(\tau)}\setminus\Omega$, then
\begin{equation}\label{maxp2} S^{(t_\star)}\cap G^{(\tau)}_\diamond=\varnothing.\end{equation}

Finally, there exists an~$s$-minimal set in~$\Omega$
coinciding with~$ G^{(t_\star)}\setminus\Omega$ outside~$\Omega$
which contains~$S^{(t_\star)}$.
\end{proposition}

\begin{proof} By assumption, we can find two different $s$-minimal sets in~$\Omega$, say~$\hat E$ and~$\hat F$,
such that~$ \hat E\setminus\Omega=\hat  F\setminus\Omega= G^{(t_\star)}\setminus\Omega$.
By Corollary~\ref{Cor1p}
(applied here with~$E_0:=F_0:=G^{(t_\star)}$, $E:=\hat E$, and~$F:=\hat F$),
we know that either~$\hat E\subseteq\hat F$ or~$\hat F\subseteq\hat E$. Up to swapping~$\hat E$ and~$\hat F$, we can assume that~$\hat F\subseteq \hat E$. Hence, since sets here are identified up to sets of null measure, we have that~$S^{(t_\star)}:=\hat E\setminus\hat F$ has positive measure.


We stress that
$$ S^{(t_\star)}\setminus\Omega=\hat E\cap\hat F^c\cap\Omega^c=(\hat E\cap \Omega^c)\cap( \hat F^c\cap\Omega^c)=
(G^{(t_\star)}\cap\Omega^c)\cap(\Omega^c\setminus G^{(t_\star)})=\varnothing$$
and therefore~$S^{(t_\star)}\subseteq\Omega$.

Let now~$t\in I$ with~$t>t_\star$ and pick an $s$-minimal set~$G^{(t)}_\diamond$ in~$\Omega$ with~$G^{(t)}_\diamond\setminus\Omega=
G^{(t)}\setminus\Omega$. It follows from Corollary~\ref{Cor1p} (applied here with~$E_0:=G^{(t_\star)}$,
$F_0:=G^{(t)}$, $E:=\hat E$, and~$F$ now equal to~$G^{(t)}_\diamond$)
that either
\begin{equation}\label{inc0a}\hat E\subseteq G^{(t)}_\diamond\end{equation}
or
\begin{equation}\label{inc0b}G^{(t_\star)}\setminus\Omega=G^{(t)}\setminus\Omega\quad{\mbox{and}}\quad
G^{(t)}_\diamond\subseteq \hat E.\end{equation}
But the first set equality in~\eqref{inc0b} cannot hold, and accordingly~\eqref{inc0a} must be satisfied.

As a byproduct of~\eqref{inc0a} we obtain that
$$ S^{(t_\star)}=\hat E\setminus \hat F\subseteq G^{(t)}_\diamond\setminus \hat F\subseteq G^{(t)}_\diamond,$$
establishing~\eqref{maxp1}.

Now we pick~$\tau\in I$ with~$ \tau<t_\star$ and we consider any~$s$-minimal set~$G^{(\tau)}_\diamond$
in~$\Omega$ with~$G^{(\tau)}_\diamond\setminus\Omega=
G^{(\tau)}\setminus\Omega$. We employ once more Corollary~\ref{Cor1p}
(here with~$E_0:=G^{(\tau)}$, $F_0:=G^{(t_\star)}$, $E:=G^{(\tau)}_\diamond$, and~$F:=\hat F$) and we gather that either
\begin{equation}\label{inc0a1}G^{(\tau)}_\diamond\subseteq \hat F\end{equation}
or
\begin{equation}\label{inc0b1}G^{(\tau)}\setminus\Omega=G^{(t_\star)}\setminus\Omega\quad{\mbox{and}}\quad
\hat F\subseteq G^{(\tau)}_\diamond.\end{equation}
But the first set equality in~\eqref{inc0b1} cannot hold, and accordingly~\eqref{inc0a1} must be satisfied, which in turn gives that~$\hat F^c\subseteq (G^{(\tau)}_\diamond)^c$.

Consequently, 
$$ S^{(t_\star)}\cap G^{(\tau)}_\diamond=\hat E\cap\hat F^c\cap G^{(\tau)}_\diamond\subseteq
\hat F^c\cap G^{(\tau)}_\diamond\subseteq (G^{(\tau)}_\diamond)^c\cap G^{(\tau)}_\diamond=\varnothing.$$
The claim in~\eqref{maxp2} is thereby proven.

The last claim in the statement of Proposition~\ref{ma:pro} is obvious, since~$
S^{(t_\star)}=\hat E\setminus\hat F\subseteq\hat E$.
\end{proof}

We observe that Proposition~\ref{ma:pro} is genuinely nonlocal, since its counterpart is false in the case
of classical minimal surfaces, as highlighted in Figure~\ref{FIGURA}.

\subsection{Completion of the proof of Theorem~\ref{gene-uni}}

We denote by~${\mathcal{I}}$ the set of the parameters~$t\in I$ for which there are at least two $s$-minimal sets
in~$\Omega$ coinciding with~$G^{(t)}\setminus\Omega$ outside~$\Omega$ and we aim at showing that~${\mathcal{I}}$ is at most countable.

To this end, for each~$t_\star\in{\mathcal{I}}$, we consider~$S^{(t_\star)}$ as in Proposition~\ref{ma:pro}.

We stress that if~$t_\star$, $t_\sharp\in{\mathcal{I}}$ and~$t_\star\ne t_\sharp$, then
\begin{equation}\label{dive}
S^{(t_\star)}\cap S^{(t_\sharp)}=\varnothing.
\end{equation}
To check this, suppose, up to exchanging them, that~$t_\star>t_\sharp$. Then, by virtue of~\eqref{maxp2},
given any $s$-minimizer~$G^{(t_\sharp)}_\diamond$ in~$\Omega$ such that~$G^{(t_\sharp)}_\diamond\setminus \Omega=G^{(t_\sharp)}\setminus \Omega$, we have that
$$S^{(t_\star)}\cap G^{(t_\sharp)}_\diamond=\varnothing.$$
Also, by the last claim in the statement of Proposition~\ref{ma:pro}, we may pick~$G^{(t_\sharp)}_\diamond$ such that~$G^{(t_\sharp)}_\diamond\supseteq S^{(t_\sharp)}$.
All in all, we have that~$S^{(t_\star)}\cap S^{(t_\sharp)}\subseteq S^{(t_\star)}\cap G^{(t_\sharp)}_\diamond=\varnothing$ and the proof of~\eqref{dive} is complete.

Now, for all~$k\in\N$, we call~${\mathcal{I}}_k$ the subset of~${\mathcal{I}}$ collecting all the parameters~$t_\star$ for which the Lebesgue measure of~$S^{(t_\star)}$ is at least~$\frac1{k+1}$.
By construction,
\begin{equation*}
{\mathcal{I}}=\bigcup_{k\in\N} {\mathcal{I}}_k,
\end{equation*}
hence the desired result is proven if we check that, for each~$k\in\N$, the set~${\mathcal{I}}_k$ is at most countable.

In fact, we see that~${\mathcal{I}}_k$ is necessarily finite, because if~$c_k\in\N\cup\{+\infty\}$ denotes its cardinality and~$|\cdot|$ is the Lebesgue measure, we deduce from~\eqref{dive} that
$$ |\Omega|\ge
\left|\bigcup_{t_\star\in {\mathcal{I}}_k} S^{(t_\star)}\right|=\sum_{t_\star\in {\mathcal{I}}_k} |S^{(t_\star)}|\ge
\sum_{t_\star\in {\mathcal{I}}_k}\frac1{k+1}=\frac{c_k}{k+1},
$$
namely~$c_k\le(k+1)|\Omega|<+\infty$.\hfill$\Box$

\subsection{More general kernels}
\label{ssec:more_general_kernels}

We conclude this section by noticing that the method presented to prove generic uniqueness works for a wide variety of kernels. In particular, the proof does not use the precise form of the
fractional kernel.  It only uses compactness of minimizers, the submodularity observation, and a strict comparison principle coming from the
positivity of the interaction between~$\Omega$ and~$\Omega^c$.

Let~$\Omega\subset\R^n$ be bounded, and let
$K:\R^n\times\R^n\to[0,+\infty]$ be a measurable symmetric kernel. We define the relative~$K$-perimeter in~$\Omega$ by
\begin{align*}
\Per_K(E,\Omega)& :=
\left( \int_{E\cap\Omega}\int_{E^c\cap \Omega} + \int_{E\cap \Omega}\int_{E^c\cap \Omega^c} + \int_{E\cap \Omega^c}\int_{E^c\cap \Omega}\right) K(x, y)\,  dy\,  dx 
\\ & =\frac12
\iint_{\R^{2n}\setminus(\Omega^c)^2}
|\chi_E(x)-\chi_E(y)|^2K(x,y)\,dx\,dy.
\end{align*}
Given~$h\in L^1(\Omega)$, we then consider the energy
\begin{equation}\label{eq:general-kernel-energy}
\mathcal F_{K,h}(E,\Omega):=\Per_K(E,\Omega)
+\int_\Omega h(x)\chi_E(x)\,dx .
\end{equation}

\begin{definition}[Compactness property]
\label{def:general-kernel-compactness}
We say that~$K$ has the compactness property in~$\Omega$ if the following
holds.  Let~$G\subset\Omega^c$ be an exterior datum, and let~$E_j$ be a
sequence of sets such that
\[
E_j\setminus\Omega=G
\qquad\text{and}\qquad
\sup_{j\in\N}\, \Per_K(E_j,\Omega)<+\infty .
\]
Then, up to a subsequence, there exists a set~$E$ with
$E\setminus\Omega=G$ such that
\[
\chi_{E_j}\to\chi_E\qquad\text{in }L^1(\Omega),
\]
and
\[
\Per_K(E,\Omega)\le \liminf_{j\to+\infty}\Per_K(E_j,\Omega).
\]
\end{definition}

We shall use the following lattice identity, for sets $E$ and $F$, whenever the quantities below are finite: 
\begin{equation}\label{eq:general-kernel-lattice-identity}
\Per_K(E,\Omega)+\Per_K(F,\Omega)
-\Per_K(E\cap F,\Omega)-\Per_K(E\cup F,\Omega)
=2\iint_{((E\setminus F)\times(F\setminus E))
\cap(\R^{2n}\setminus(\Omega^c)^2)}\hspace{-1cm}K(x,y)\,dx\,dy.
\end{equation}
In particular,
\begin{equation}\label{eq:general-kernel-submodularity}
\Per_K(E\cap F,\Omega)+\Per_K(E\cup F,\Omega)
\le \Per_K(E,\Omega)+\Per_K(F,\Omega).
\end{equation}
The lower-order term is exactly modular as well:
\begin{equation}\label{eq:general-kernel-modular-term}
\int_\Omega h\chi_E+
\int_\Omega h\chi_F
=
\int_\Omega h\chi_{E\cap F}
+
\int_\Omega h\chi_{E\cup F}.
\end{equation}

\begin{proposition} 
\label{prop:general-kernel-extremal-minimizers}
Assume that~$K$ has the compactness property in~$\Omega$, and let
$h\in L^1(\Omega)$.  Let~$G\subset\Omega^c$ be an exterior datum for which
there exists an admissible set~$E$ with
\[
E\setminus\Omega=G
\qquad\text{and}\qquad
\Per_K(E,\Omega)<+\infty .
\]
Then, the minimization problem
\[
\inf\big\{\mathcal F_{K,h}(E,\Omega):E\setminus\Omega=G\big\}
\]
admits a minimizer.  

Moreover, if~$\mathcal M_G$ denotes the family of all
minimizers, then there exist $E_G^-, E_G^+\in \mathcal{M}_G$ such that
\[
E_G^-\subseteq E\subseteq E_G^+
\qquad\text{in }\Omega, \quad \text{for every}\ E\in\mathcal M_G.
\]
\end{proposition}

\begin{proof}
Let~$-\|h\|_{L^1(\Omega)}\le m_G < +\infty$ be the infimum, and let~$(E_j)_{j\in \N}$ be a minimizing sequence, which by the compactness property (Definition~\ref{def:general-kernel-compactness}) satisfies
\[
\chi_{E_j}\to\chi_E\qquad\text{in }L^1(\Omega)
\]
up to a subsequence, for some~$E$ with~$E\setminus\Omega=G$, and
$$
\Per_K(E,\Omega)\le\liminf_{j\to+\infty}\Per_K(E_j,\Omega).
$$
The convergence is a.e. in~$\Omega$ (up to a subsequence) and dominated convergence therefore implies that $E$ is, in fact, a minimizer. 

By~\eqref{eq:general-kernel-submodularity} and~\eqref{eq:general-kernel-modular-term},  the class of minimizers is closed under unions and
intersections.  In particular, taking unions of a maximizing sequence in $\mathcal M_G$ for $|E\cap \Omega|$ yields the existence of some $E_G^+$ as above, with 
\[
|E_G^+\cap \Omega| =  \sup\{|E \cap \Omega| : E\in \mathcal M_G\} \ge |(E_G^+\cup E) \cap \Omega|\quad\text{for any}\  E \in \mathcal M_G. 
\]
This implies $E \subseteq E_G^+$ in $\Omega$, as we wanted. The construction for $E^-_G$ is analogous with intersections and the infimum.
\end{proof}

We can now prove the generic uniqueness statement for general perimeters:

\begin{proposition}[Generic uniqueness for general kernels]
\label{prop:general-kernel-generic-uniqueness}
Assume that~$K$ has the compactness property in~$\Omega$ (Definition~\ref{def:general-kernel-compactness}), and that
\begin{equation}\label{eq:interior-exterior-positivity}
K(x,y)>0
\qquad\text{for a.e.}\ (x,y)\in\Omega\times\Omega^c.
\end{equation}
Let~$h\in L^1(\Omega)$.  Let~$I\subset\R$, and let
$\{G^{(t)}\}_{t\in I}$ be a family of exterior data in~$\Omega^c$.
Assume that, for each~$t\in I$, there is an admissible set of finite
$K$-perimeter with exterior datum~$G^{(t)}$.  

Assume also that, whenever
$\tau,t\in I$ and~$\tau<t$,
\begin{equation}\label{eq:strict-exterior-data-general-kernel}
G^{(\tau)}\subsetneq G^{(t)}
\qquad\text{in }\Omega^c.
\end{equation}
Then the set of parameters~$t\in I$ for which the minimizer of
\eqref{eq:general-kernel-energy} with exterior datum~$G^{(t)}$ is not unique
is at most countable.
\end{proposition}

\begin{proof}
By Proposition~\ref{prop:general-kernel-extremal-minimizers}, for each
$t\in I$ there are a minimal and maximal minimizer~$E_t^-$, $E_t^+$, 
associated to the exterior datum~$G^{(t)}$.

Let~$\tau<t$, let $E$ be a minimizer with exterior datum~$G^{(\tau)}$, and let~$F$ be a
minimizer with exterior datum~$G^{(t)}$: the set~$E\cap F$ is
admissible for the datum~$G^{(\tau)}$, while~$E\cup F$ is admissible for $G^{(t)}$.  Hence, minimality gives
\[
\mathcal F_{K,h}(E,\Omega)+\mathcal F_{K,h}(F,\Omega)
\le
\mathcal F_{K,h}(E\cap F,\Omega)
+
\mathcal F_{K,h}(E\cup F,\Omega).
\]
The opposite inequality follows from \eqref{eq:general-kernel-submodularity}-\eqref{eq:general-kernel-modular-term}, so that equality holds. In particular,
\begin{equation}\label{eq:vanishing-defect-general-kernel}
\iint_{((E\setminus F)\times(F\setminus E))
\cap(\R^{2n}\setminus(\Omega^c)^2)}K(x,y)\,dx\,dy =0.
\end{equation}

We claim that~$E\subseteq F$ in~$\Omega$.  Otherwise, 
\[
A:=(E\setminus F)\cap\Omega
\]
has positive measure.  

On the other hand,
\[
B:=G^{(t)}\setminus G^{(\tau)} = (F\setminus E)\cap \Omega^c
\]
has positive measure by~\eqref{eq:strict-exterior-data-general-kernel} (the strict inclusion is up to sets of measure zero). 

Since~$K>0$ a.e. on
$\Omega\times\Omega^c$, we have
\[
\int_A\int_B K(x,y)\,dy\,dx>0,
\]
which contradicts~\eqref{eq:vanishing-defect-general-kernel}, since  $A\times B\subset\R^{2n}\setminus(\Omega^c)^2$.  

In particular, we have shown that
\begin{equation}\label{eq:extremal-order-general-kernel}
E_\tau^+\subseteq E_t^-
\qquad\text{in }\Omega
\qquad\text{whenever }\tau<t.
\end{equation}
For each~$t\in I$, define the gap
\[
S_t:=(E_t^+\setminus E_t^-)\cap \Omega.
\]
If~$\tau<t$, then~\eqref{eq:extremal-order-general-kernel} gives
\[
S_\tau\subseteq E_\tau^+\cap \Omega\subseteq E_t^-\cap \Omega\qquad{\mbox{and}}
\qquad
S_t\subseteq \Omega\setminus E_t^-.
\]
Thus~$S_\tau\cap S_t=\varnothing$.  Hence the family~$\{S_t\}_{t\in I}$ is
pairwise disjoint in the finite-measure set~$\Omega$.

The minimizer at time~$t$ is not unique if and only if~$|S_t|>0$.  Since
there are at most countably many pairwise disjoint measurable subsets of
$\Omega$ with positive measure, the set of such parameters is at most
countable.
\end{proof}

The result in Proposition~\ref{prop:general-kernel-generic-uniqueness} holds, in particular, in the following cases: 

\begin{lemma}[Examples of kernels]
\label{lem:general-kernel-examples}
Let~$\Omega\subset\R^n$ be a bounded Lipschitz domain.  Let~$K$ be measurable,
symmetric, and nonnegative.  Assume that
\begin{equation}\label{eq:finite-interaction-general-kernel}
\int_\Omega\int_{\Omega^c} K(x, y)\, dy\, dx <+\infty,
\end{equation}
that
\begin{equation}\label{eq:positivity-example-general-kernel}
K(x,y)>0
\qquad\text{for a.e. }(x,y)\in\Omega\times\Omega^c,
\end{equation}
and that there exist~$s\in(0,1)$,~$r_0>0$, and~$\lambda>0$ such that
\begin{equation}\label{eq:local-coercivity-general-kernel}
K(x,y)\ge \lambda |x-y|^{-n-s}\quad\text{for a.e. }x, y\in \Omega\ \text{with}\ 0<|x-y|< r_0.
\end{equation}
Then~$K$ has the compactness property in~$\Omega$.  Moreover, every exterior
datum~$G\subset\Omega^c$ admits a finite-energy admissible competitor.
Consequently, Proposition~\ref{prop:general-kernel-generic-uniqueness}
applies to the energy~\eqref{eq:general-kernel-energy}, for every
$h\in L^1(\Omega)$ and every family of exterior data satisfying
\eqref{eq:strict-exterior-data-general-kernel}.
\end{lemma}

\begin{proof}
Let~$G\subset\Omega^c$, and let~$E_j$ be a sequence such that
$E_j\setminus\Omega=G$ and
\[
\sup_{j\in\N}\Per_K(E_j,\Omega)<+\infty.
\]
By the local lower bound~\eqref{eq:local-coercivity-general-kernel},
\[
\iint_{\substack{(\Omega\times\Omega)\cap  \{|x-y|<r_0\}}}
\frac{|\chi_{E_j}(x)-\chi_{E_j}(y)|}{|x-y|^{n+s}}\,dx\,dy
\le C\Per_K(E_j,\Omega).
\]
The remaining part of the~$W^{s,1}(\Omega)$ seminorm, corresponding to
$|x-y|\ge r_0$, is bounded by a constant depending only on~$\Omega$,~$r_0$,
and~$s$.  Hence~$\chi_{E_j}$ is bounded in~$W^{s,1}(\Omega)$.  Since
$\Omega$ is bounded and Lipschitz, the embedding
$W^{s,1}(\Omega)\hookrightarrow L^1(\Omega)$ is compact.  Thus, up to a
subsequence,
\[
\chi_{E_j}\to\chi_E
\qquad\text{in }L^1(\Omega)
\]
for some set~$E$ with~$E\setminus\Omega=G$.  Passing to a further subsequence,
we may assume a.e. convergence in~$\Omega$.  Fatou's lemma and the
nonnegativity of~$K$ give
\[
\Per_K(E,\Omega)\le\liminf_{j\to+\infty}\Per_K(E_j,\Omega).
\]
Thus~$K$ has the compactness property.

Assumption~\eqref{eq:finite-interaction-general-kernel}
provides a finite-energy competitor.  Indeed, if
$E=\Omega\cup G$, then
\[
\Per_K(E,\Omega)
=\int_\Omega \int_{\Omega^c\setminus G} K(x, y)\, dy\, dx 
\le \int_\Omega \int_{\Omega^c} K(x, y)\, dy\, dx<+\infty.
\]
This completes the proof.
\end{proof}

\begin{remark}{\rm
In particular, the assumptions in Lemma~\ref{lem:general-kernel-examples} are satisfied by the fractional kernels
\[
K(x,y)=\frac{a(x,y)}{|x-y|^{n+s}},
\]
and by the tempered fractional kernels
\[
K(x,y)=a(x,y)\frac{e^{-m|x-y|}}{|x-y|^{n+s}},
\qquad m\ge0,
\]
provided~$s\in(0,1)$,
\[
a(x,y)=a(y,x),
\qquad
0<\lambda\le a(x,y)\le\Lambda<+\infty
\]
for a.e.~$x,y\in\R^n$. 

General kernels for nonlocal minimal surfaces have been studied,
for example, in~\cite{MR3981295}. See also~\cite{MR3401008}
for the corresponding parabolic problem.
}\end{remark}

\section{Generic regularity: proof of
Theorems~\ref{thm:main1} and \ref{thm:main2}}\label{sec:generic-regularity}

The goal of this section is to prove the generic regularity result stated in Theorems~\ref{thm:main1} and \ref{thm:main2},
constructing a family of arbitrarily small perturbations of the exterior datum (in a suitable sense) giving rise to smooth minimizers in one dimension more than guaranteed by the general theory of regularity.

The blueprint of the argument takes inspiration from~\cite{2023arXiv230813209F}, but the specific arguments need to be adapted to the peculiar case of nonlocal minimal surfaces.

\subsection{A monotone family of exterior data}
We select a family of exterior data as an arbitrarily small perturbation of a given exterior datum $G^{(0)}$, in such a way that we can ensure a strict quantitative separation of minimizers.

To accomplish this goal, we let $G^{(0)}\subseteq B_1^c$ and set
\[
A_0:=B_1\cup G^{(0)}.
\]
For $q\in[0,1]$, we define
\begin{equation}\label{def-special-family.prre}
D_q(A_0):=(1+q)A_0+B_q,
\end{equation}
and, for $t\in[0,1]$,
\begin{equation}\label{def-special-family}
A_t:=\bigcup_{0\le q\le t}D_q(A_0),
\qquad
G^{(t)}:=A_t\setminus B_1.
\end{equation}
Then $A_t=B_1\cup G^{(t)}$, and we will assume that the set $G^{(0)}$ is such that the family $\{G^{(t)}\}_{t\in[0,1]}$ is (strictly) increasing (at least, for $t$ close to 0).  For every $t$, we let $E^{(t)}$ be an $s$-minimal set in $B_1$ with exterior datum $G^{(t)}$, which satisfies (by Corollary~\ref{Cor1p})
\begin{equation}\label{nested-maximal-branch}
E^{(\tau)}\subseteq E^{(t)}
\qquad\text{whenever }0\le\tau<t\le1.
\end{equation}
 Theorem \ref{gene-uni} implies that $E^{(t)}$ is the unique minimizer for all but countably many $t \ge 0$. 

The following result gives the strict quantitative inclusion of minimizers with this exterior family of boundary data (and it is the counterpart of~\cite[Lemma 4.3]{2023arXiv230813209F} in our setting). We recall that we deal with representatives of sets up to removing sets of measure zero, and with essential boundaries. 

\begin{lemma}\label{lem:linear-retraction}
Let $0\le\tau<t<1$ and $\eta\in(0,1)$.  Then,
\begin{equation}\label{linear-retraction}
B_{\frac\eta2(t-\tau)}(x)\subseteq E^{(t)}
\qquad\text{for every }x\in\partial E^{(\tau)}\cap B_{1-\eta}.
\end{equation}
\end{lemma}

\begin{proof}
Let us denote
\[
    \lambda:=1+\frac{t-\tau}{1+\tau}\in (1, 2),\qquad r:=\lambda-1=\frac{t-\tau}{1+\tau}\in (0, 1).
\]
We first notice that
\begin{equation}\label{PRi}
    \lambda A_\tau+B_r\subseteq A_t.
\end{equation}
Indeed, let $q\in[0,\tau]$ and set $q':=\lambda(1+q)-1$. Then, we see that $0\le q'\le t$, and, since $q'=\lambda q+r$,
\[
    \lambda D_q(A_0)+B_r =\lambda(1+q)A_0+B_{\lambda q+r} =(1+q')A_0+B_{q'}
      =D_{q'}(A_0) \subseteq A_t.
\]

Pick now a vector $v\in B_r$ and define
\[
    T_v(z):=\lambda z+v.
\]
Since $\lambda>1+|v|$, using \eqref{PRi} we see that
\[
    B_1\subseteq T_v(B_1)=B_\lambda(v)\qquad{\mbox{and}}\qquad T_v(A_\tau)\subseteq A_t.
\]
 
By scale invariance,  $T_v(E^{(\tau)})$ is minimizing  in
$T_v(B_1)$, and hence it is minimizing in $B_1$ with exterior datum
\[
     T_v(E^{(\tau)})\setminus B_1\subsetneq G^{(t)}.
\]
Hence, 
\[
    T_v(E^{(\tau)}) \subseteq E^{(t)}    \qquad\forall v\in B_r.
    \]

In particular, we have shown that
\[
    B_r(\lambda x)\subseteq E^{(t)},\qquad\forall x\in \partial E^{(\tau)}
\]

Finally, for $x\in \partial E^{(\tau)}\cap B_{1-\eta}$, since $|\lambda x - x|\le r(1-\eta)$
we have 
\[
    B_{\eta r}(x)\subseteq B_r(\lambda x)\subseteq E^{(t)}.
\]
Since $r \ge (t-\tau) / 2$, we are done.
\end{proof}

This gives:
\begin{corollary}\label{cor:space-time-separation}
If $0\le \tau< t<1$, then
\begin{equation}\label{eq:space-time-separation.0}
\partial E^{(\tau)}\cap\partial E^{(t)}\cap B_1=\varnothing.
\end{equation}
Moreover, for $\eta\in(0,1)$, if
\[
x\in\partial E^{(\tau)}\cap B_{1-\eta},
\qquad
y\in\partial E^{(t)},
\qquad \tau<t,
\]
then
\begin{equation}\label{eq:space-time-separation}
|x-y|\ge\frac\eta2(t-\tau).
\end{equation}
\end{corollary}
 
\subsection{Conical results} 
Let us now state a sequence of preliminary results on general cones and $s$-minimal cones to be used in the following: 

The first result is straightforward from the definition of cone (we include the simple proof for completeness):

\begin{lemma}\label{lem:homogeneous-inclusinos}
Let $C,D\subseteq\R^n$ be cones and let $e\in\R^n$.
\begin{enumerate}
\item If $C-e\subseteq D$, then $C-\rho e\subseteq D$ for every $\rho>0$.
\item If $C-e\subseteq D$ and $D$ is closed, then $C\subseteq D$.
\end{enumerate}
\end{lemma}

\begin{proof} To prove the first claim, we observe that
$$ C-\rho e=\rho\left(\frac{C}\rho-e\right)=\rho(C-e)\subseteq\rho D=D.$$

To prove the second claim, we notice that, for all $x\in C$ and~$\rho>0$,
by the first claim we have that~$x-\rho e\in D$ and, consequently,
$$x=\lim_{\rho\searrow0}x-\rho e\in\overline{D}=D,$$ as advertised.
\end{proof}

Next is a useful observation about conical graphs.

\begin{lemma}\label{GRPH}
Let $C$ be a cone, $e\in\mathbb S^{n-1}$ and~$r>0$.  Assume that
\begin{equation}\label{monotone-cone-inclusion.0}
C-e\subseteq C.
\end{equation}
Assume also that
\begin{equation}\label{monotone-cone-inclusion.02}
B_r(-e)\subseteq C\qquad{\mbox{and}}\qquad B_r(e)\subseteq C^c.
\end{equation}

Then, $C$ has a graphical structure in the $e$-direction.
\end{lemma}

\begin{proof} Up to a rotation, we suppose that~$e=e_n:=(0,\dots,0,1)$. By~\eqref{monotone-cone-inclusion.0}
and Lemma \ref{lem:homogeneous-inclusinos}(1), we know that
\begin{equation}\label{monotone-cone-inclusion.03}
C-\R_+ e_n\subseteq C.\end{equation}

Also, by~\eqref{monotone-cone-inclusion.02},
\begin{equation*}
\left\{ x_n>\frac{|x'|}r\right\}=\R_+\left\{ x_n=1>\frac{|x'|}r\right\}
\subseteq \R_+ B_r(e_n)\subseteq\R_+ C^c=C^c
\end{equation*}
and, similarly,
\begin{equation*}
\left\{ x_n<-\frac{|x'|}r\right\}=\R_+\left\{ x_n=-1<-\frac{|x'|}r\right\}
\subseteq \R_+ B_r(-e_n)\subseteq\R_+ C=C.
\end{equation*}
As a consequence,
\begin{equation*}
\partial C\subseteq \left\{ |x_n|\le\frac{|x'|}r\right\}.
\end{equation*}

Furthermore, by~\eqref{monotone-cone-inclusion.0}, we know that~$C$ is a subgraph, namely
there exists~$u:\R^{n-1}\to\R$ such that
\begin{equation*} \partial C=\big\{ (x',u(x')),\; x'\in\R^{n-1}\big\},\end{equation*}
with
\begin{equation*}
u(x'):=\sup\left\{ t\in\R {\mbox{ s.t. $ (x',\tau)\in C$ for all~$\tau<t$}}\right\}
\in\left[-\frac{|x'|}r,\frac{|x'|}r\right].
\qedhere\end{equation*}
\end{proof}


The following is the analogue of \cite[Lemma 4.2]{2023arXiv230813209F}
for nonlocal minimal cones (monotone cones are either halfspaces or invariant in the direction of monotonicity):

\begin{lemma}\label{lem:monotone-cones}
Let $C$ be an $s$-minimal cone and $e\in\mathbb S^{n-1}$.  Assume that
\begin{equation}\label{monotone-cone-inclusion}
C-e\subseteq C.
\end{equation}
Then either $C$ is $e$-invariant (namely,
$C+\R e=C$),
or $C$ is a halfspace.
\end{lemma}

\begin{proof}
 We can assume that~$C\ne\R^n$, otherwise we are done. Suppose that
\begin{equation}\label{2wsPSK0eoj2orm23oifujevw0}
{\mbox{either~$e$ or $-e$ belongs to $\partial C$.}}
\end{equation}
If~$e\in\partial C$, then $0=e-e\in\partial(C-e)$, and so $C$ and~$C-e$ touch at the origin. If~$-e\in\partial C$, then $-e=0-e\in\partial(C-e)$, and they touch at~$-e$.
In both cases, $C-e\subseteq C$ and the strict maximum principle of~\cite{2023arXiv230801697D} gives
\[
C-e=C, 
\]
and hence, also~$C+e=C$. 

Thus, applying Lemma~\ref{lem:homogeneous-inclusinos}-(1) to the inclusions~$C-e\subseteq C$ and~$C+e\subseteq C$, we obtain
\[
C+\R e\subseteq C
 \subseteq C+\R e.
\]
That is, $C$ is $e$-invariant.

Accordingly, we now suppose that~\eqref{2wsPSK0eoj2orm23oifujevw0} is violated,
namely~$e\not\in\partial C$ and~$-e\not\in\partial C$. Thus, since, by~\eqref{monotone-cone-inclusion},
$$ -e=0-e\in\overline{C}-e=\overline{C-e}\subseteq\overline{C}
$$
and
$$ e=0+e\in\overline{C^c}+e=\overline{C^c+e}=
\overline{(C+e)^c}\subseteq\overline{C^c},$$
we conclude that~$-e$ lies in the interior of $C$ and~$e$ lies in the interior of $C^c$. 

This allows us to use Lemma~\ref{GRPH} and conclude that~$C$
is a nonlocal minimal graph. But then, by~\cite{MR3934589}, we obtain that~$\partial C$ is a smooth graph. The conical structure of~$C$ thereby
gives that~$C$ is a halfspace.
\end{proof}

\subsection{Blow-up across times}  Let us recall the
monotonicity formula of~\cite{MR2675483}.

{F}rom the procedural point of view,  given a set~$E\subseteq\R^n$,
a point~$x\in\partial E$, and~$r\in(0,+\infty)$,
what we need is  the existence of
some ``quantity'', say~$W_s(E, x, r)=W_s(E-x,0,r)$, which:
\begin{itemize}
\item[(P1)] is increasing in~$r$,
\item[(P2)] is constant in~$r$ if and only if $E-x$ is a cone,
\item[(P3)] for all~$\lambda>0$, it holds that~$
W_s(\lambda E,0,r)=W_s\left(E,0,\frac{r}\lambda\right)$,
\item[(P4)] $W_s(E_k,x_k,r) \to W_s(E,x_0,r)$ as~$k \to +\infty$ whenever~$ E_k \to E$
locally in~$L^1(\R^n)$
and~$x_k \to x_0$.
\end{itemize}

Thus, if~$\Phi_E(r)$ is the function introduced at the beginning of Section~8
in~\cite{MR2675483}, we define
$$ W_s(E,x,r):=\Phi_{E-x}(r)$$
and we have that~(P1) above follows from~\cite[Theorem 8.1]{MR2675483},
that~(P2) is guaranteed by~\cite[Corollary~8.2]{MR2675483},
that~(P3) is a consequence of scaling (see the proof of Theorem~9.2
in~\cite{MR2675483}),
and~(P4) follows from the last claim in~\cite[Proposition 9.1]{MR2675483}
(combined with the continuity of the translations in Lebesgue spaces).

The monotonicity of~$W_s$ thus allows one to define
\[
\omega_s(E,x):=\lim_{r\searrow0}W_s(E,x,r).
\]

The above setting also permits the following extension of~(P4):

\begin{lemma}\label{LE4724rf}
Suppose that~$ E_k \to E$
locally in~$L^1(\R^n)$,
that~$x_k \to x_0$, and that~$r_k\searrow0$, as~$k \to +\infty$.

Assume also that
\begin{equation}\label{AGGome}
\liminf_{k\to+\infty} \omega_s(E_k,x_k)\ge \omega_s(E,x_0).\end{equation}

Then, $W_s(E_k,x_k,r_k) \to \omega_s(E,x_0)$.
\end{lemma}

\begin{proof} Let~$\epsilon>0$ and suppose that~$k$ is sufficiently large such that~$r_k\in(0,\epsilon)$. Then, using~(P1),
we know that~$W_s(E_k,x_k,r_k)\le W_s(E_k,x_k,\epsilon)$ and accordingly, by~(P4),
$$ \limsup_{k\to+\infty} W_s(E_k,x_k,r_k)\le \limsup_{k\to+\infty} W_s(E_k,x_k,\epsilon)
=W_s(E,x_0,\epsilon).$$
Taking now~$\epsilon\searrow0$, we obtain
\begin{equation}\label{SOojlqmsBMS}
\limsup_{k\to+\infty} W_s(E_k,x_k,r_k)\le\omega_s(E,x_0).
\end{equation}

Moreover, using again~(P1), for all~$k\in\N$ and~$\delta\in(0,r_k)$,
we have that~$W_s(E_k,x_k,r_k)\ge W_s(E_k,x_k,\delta)$ and therefore~$
W_s(E_k,x_k,r_k)\ge \omega_s(E_k,x_k)$.
We now use~\eqref{AGGome} and conclude that
$$\liminf_{k\to+\infty}W_s(E_k,x_k,r_k)\ge
\liminf_{k\to+\infty} \omega_s(E_k,x_k)\ge \omega_s(E,x_0).$$
This and~\eqref{SOojlqmsBMS} yield the desired result.
\end{proof}

This setting entails the following result, which can be
considered the analogue of \cite[Lemma~4.5]{2023arXiv230813209F}
for nonlocal minimal surfaces:

\begin{lemma} \label{lem:variable-center-blowup}
Let
\[
(x_k,t_k)\to (x_0,t_0), \qquad x_k\in\partial E^{(t_k)}, \qquad
x_0\in\partial E^{(t_0)}\cap B_1,
\]
 with $x_k \neq x_0$. Assume that, as $k\to+\infty$,
\begin{equation}\label{omega-convergence}
\omega_s(E^{(t_k)},x_k)\to \omega_s(E^{(t_0)},x_0),
\end{equation}
and that $E^{(t_k)}\to E^{(t_0)}$ locally in $L^1(\R^n)$.  

Then, up to a subsequence,
\[
\frac{E^{(t_k)}-x_k}{|x_k - x_0|}\to  C
\qquad\text{locally in }L^1(\R^n),
\]
where $C$ is an $s$-minimal cone.  

Additionally,
\begin{equation}\label{SINGO}
{\mbox{if each $x_k\in {\rm Sing}(\partial E^{(t_k)})$, then $0\in{\rm Sing}(\partial C)$.}}\end{equation}
\end{lemma}

\begin{proof}
The proof is a technical modification of the one in \cite[Lemma 4.5]{2023arXiv230813209F}
(for instance, the assumption of uniform convergence in \cite[Lemma 4.5]{2023arXiv230813209F} needs to be replaced here by
local convergence in~$L^1(\R^n)$).

Let~$r_k:=|x_k-x_0|$ and~$F_k:=\frac{E^{(t_k)}-x_k}{|x_k - x_0|}$.
We remark that~$F_k$ is $s$-minimal in~$B_{(1-|x_0|)/2r_k}$ for $k$ large and, therefore, given~$R>0$,
for large~$k$, we have that
$B_{(1-|x_0|)/2r_k}\supseteq B_{2R}$ and
\[
\Per_s(F_k,B_{2R})
\le
\Per_s(F_k\cup B_{2R},B_{2R}) \le \Per_s(B_{2R},B_{2R})
\le C_R .
\]
 
Hence
\[
\iint_{B_R\times B_R}
\frac{|\chi_{F_k}(x)-\chi_{F_k}(y)|}{|x-y|^{n+s}}\,dx\,dy
\le C_R .
\]
Since also $\|\chi_{F_k}\|_{L^1(B_R)}\le |B_R|$, the sequence
$\chi_{F_k}$ is bounded in $W^{s,1}(B_R)$. By the compact embedding
$W^{s,1}(B_R)\Subset L^1(B_R)$, up to a subsequence $\chi_{F_k}$ converges locally in~$L^1(\R^n)$.

We denote by~$C$ this limit. We know that~$C$
is $s$-minimal (by~\cite[Theorem~3.3]{MR2675483}).

Notice also that condition~\eqref{AGGome} is fulfilled, thanks to~\eqref{omega-convergence}, therefore, by Lemma~\ref{LE4724rf}, for all~$r>0$ we have that
\begin{equation}\label{asPrH}\lim_{k\to+\infty}W_s(E^{(t_k)},x_k,r_k r) = \omega_s(E^{(t_0)},x_0).\end{equation}
Moreover, in view of~(P3),
$$W_s(F_k,0,r)=W_s\left( \frac{E^{(t_k)}-x_k}{r_k},0,r\right)=
W_s(E^{(t_k)}-x_k,0,r_kr)=W_s(E^{(t_k)},x_k,r_kr).
$$
It follows from this and~\eqref{asPrH} that
$$\lim_{k\to+\infty}W_s(F_k,0,r)=\lim_{k\to+\infty}W_s(E^{(t_k)},x_k,r_kr)=\omega_s(E^{(t_0)},x_0).
$$
Hence, by~(P4),
$$W_s(C,0,r)=\omega_s(E^{(t_0)},x_0),
$$
and we stress that the latter quantity is independent of~$r$. Consequently, by~(P2), we conclude that~$C$ is a cone, as desired.

It remains only to prove~\eqref{SINGO}. To accomplish this goal, suppose that~$E^{(t_k)}$ is singular at~$x_k$ (and accordingly~$F_k$ is singular at the origin), but, for the sake of contradiction, that~$C$ is smooth at the origin. By uniform density estimates (see~\cite[Theorem~4.1]{MR2675483}), we know that~$C\ne\varnothing$ and~$C\ne\R^n$, therefore necessarily~$C$
is a halfspace. Then, by~\cite[Corollary~4.4(ii)]{MR2675483},
we have that~$F_k\cap B_1$ lies in an arbitrarily small neighborhood
of the halfspace~$C$
as soon as~$k$ is large enough. 

Consequently, by the improvement of flatness result in~\cite[Theorem~6.1]{MR2675483}, one finds that~$\partial F_k\cap B_{1/2}$ does not contain singular points,
in contradiction with our assumptions. 
\end{proof}

\subsection{Proof of generic regularity} We introduce  the set of singular points (in ``space-time''): 
\[
\cS:=\big\{(x,t)\in B_1\times(0,1):x\in{\rm Sing}(\partial E^{(t)})\big\}.
\]
We denote by $\cI$ the set of nonuniqueness times, i.e. the collection of~$t\in(0,1)$ for which the minimizer~$E^{(t)}$ is not unique.
We have, in light of Theorem~\ref{gene-uni}, that~$\cI$ is at most countable.

We also denote by~$\cD$ the set of discontinuity times for the
local convergence in~$L^1(\R^n)$, that is we define $$\mathcal{T}:= 
\big\{t\in(0,1) \mbox{ s.t. $
E^{(t_k)}\to E^{(t)}$ locally in $L^1(\R^n)$
whenever~$ t_k \to t$}\big\} $$
and~$\cD:=(0,1)\setminus\mathcal{T}$.
We have that:

\begin{lemma}
\label{lem:Dcountable}
$\cD$ is at most countable.
\end{lemma}
\begin{proof} For all $m\in\N$, we define $$\mathcal{T}_m:= 
\big\{t\in(0,1) \mbox{ s.t. $
E^{(t_k)}\to E^{(t)}$ in $L^1(B_m)$
whenever~$ t_k \to t$}\big\} $$
and~$\cD_m:=(0,1)\setminus\mathcal{T}_m$.
Since the countable union of countable sets is countable,
it suffices to show that each~$\cD_m$ is countable.

For this, for $t\in(0,1)$ we define $\varphi_m(t):=|E^{(t)}\cap B_m|$ and let~${\mathcal{E}}_m$ be the set of
discontinuity points for~$\varphi_m$. 

By~\eqref{nested-maximal-branch}, we know that~$E^{(t)}$ is an increasing family of sets with respect to~$t$ with respect to the notion of set inclusion. Accordingly, 
the function $\varphi_m$ is monotone.
Since monotone increasing functions have countably many discontinuities, we gather that~${\mathcal{E}}_m$ is countable.

Then, the desired result would be established if we knew that
\begin{equation}\label{nmon0MnaT}
\cD_m\subseteq{\mathcal{E}}_m.
\end{equation}
We now check this. Let $t\in\cD_m$. Then, there exist $c>0$ and a sequence $t_k$ converging to $t$ such that $|(E^{(t_k)}\triangle E^{(t)})\cap B_m|\ge c$
for all~$k\in\N$ (here, the notation ``$\triangle$'' stands for symmetric difference of sets). Hence,
\begin{eqnarray*} c&\le& |(E^{(t_k)}\setminus E^{(t)})\cap B_m|+|(E^{(t)}\setminus E^{(t_k)})\cap B_m|\\&=&\begin{cases}
|(E^{(t_k)}\setminus E^{(t)})\cap B_m|&{\mbox{ if }}t_k>t,\\
|(E^{(t)}\setminus E^{(t_k)})\cap B_m|&{\mbox{ if }}t_k<t,
\end{cases}\\&=&\begin{cases}
\varphi_m(t_k)-\varphi_m(t)&{\mbox{ if }}t_k>t,\\
\varphi_m(t)-\varphi_m(t_k)&{\mbox{ if }}t_k<t,
\end{cases}\\&=&|\varphi_m(t_k)-\varphi_m(t)|.
\end{eqnarray*}
The proof of \eqref{nmon0MnaT} is thereby complete.
\end{proof}

Now we let 
$$\cS':=\big\{ {\mbox{$(x,t)\in\cS$ s.t. $t\not\in(\cI\cup\cD)$}}\big\}$$
and consider the projections~$\pi_x(x,t):=x$ and~$\pi_t(x,t):=t$.

It is useful to write the set of singular points as a graph:
 
\begin{lemma}\label{SIGRA}
For each $x\in\pi_x(\cS)$ there exists a unique~$\tau(x)$ such that~$(x,\tau(x))\in\cS$. Moreover, the function $\tau:\pi_x(\cS)\to(0,1)$ is locally Lipschitz.
\end{lemma}

\begin{proof} Suppose, for the sake of contradiction, that~$(x,t_1)$, $(x,t_2)\in\cS$, for some~$t_1\ne t_2\in(0,1)$. Then, 
\begin{equation}\label{v567BUAJSK9}
x\in {\rm Sing}(\partial E^{(t_1)})\cap{\rm Sing}(\partial E^{(t_2)})\cap B_1\subseteq
\partial E^{(t_1)}\cap \partial E^{(t_2)}\cap B_1.\end{equation}

However, by~\eqref{eq:space-time-separation.0} in Corollary~\ref{cor:space-time-separation}, we have that~$\partial E^{(t_1)}\cap\partial E^{(t_2)}\cap B_1=\varnothing$, which is in contradiction with~\eqref{v567BUAJSK9}. 

To see $\tau$ is Lipschitz, let~$x\in\pi_x(\cS)\subseteq B_1$ and~$\rho\in(0,1-|x|)$. Let also~$y\in \pi_x(\cS)\cap B_\rho(x)\subseteq B_{\rho+|x|}$. We set~$\eta:=\min\{1-|x|,\,1-\rho-|x|\}$ and we remark that~$x$, $y\in B_{1-\eta}$. Thus, by~\eqref{eq:space-time-separation} in Corollary~\ref{cor:space-time-separation}, we see that~$|x-y|\ge\frac\eta2|\tau(x)-\tau(y)|$.
\end{proof}

Define now
\begin{equation}
    \label{eq:fdef}
f(x):=\omega_s(E^{(\tau(x))},x).
\end{equation}
Let us also denote by $\cS^*$, 
\begin{equation}\label{destar}
\cS^* := \big\{(x, t) \in \cS' : \exists (x_k, t_k) \in \cS'\setminus\{(x, t)\}\ \text{with}\ (x_k, t_k)\to (x, t)\ \text{and}\ f(x_k)\to f(x) \big\}, 
\end{equation}
that is, the collection of points~$(x,\tau(x))\in \cS'$ for which
there exists a sequence~$x_k \in\pi_x(\cS')\setminus\{x\}$ such that~$x_k \to x$ and~$f(x_k) \to f(x)$
as~$k\to+\infty$. By \cite[Lemma 2.28]{2023arXiv230813209F} (which corresponds to~\cite[Lemma 7.1]{MR4179834}), we have
that~
\begin{equation}
    \label{eq:pitS's*}
    \text{$\pi_t(\cS'\setminus\cS^*)$ is at most countable.}
\end{equation}

The work done so far allows us to obtain the following result, which is the analogue of \cite[Proposition~4.6]{2023arXiv230813209F}
for nonlocal minimal surfaces (recall that~$n_s^*$ was introduced in~\eqref{DIMECR}):

\begin{proposition}\label{prop:critical-generic-regularity}
Let $n=n_s^*+1$.  Then
\[
\cS^*=\varnothing.
\]
\end{proposition}

\begin{proof}
Suppose, by contradiction, that~$\cS^*$ is not empty. Let~$(x_0,t_0)\in\cS^*$. In particular, by Lemma~\ref{SIGRA}, we have~$t_0=\tau(x_0)$. 

Then,~$x_0\in{\rm Sing}(\partial E^{(t_0)})$ and, by the definition of~$\cS^*$
in~\eqref{destar}, there exist points~$x_k\in\pi_x(\cS')\setminus\{x_0\}$ such that~$x_k\to x_0$ as $k\to+\infty$ and
\begin{equation}\label{den-conv-critical}
\omega_s(E^{(\tau(x_k))},x_k)\to \omega_s(E^{(t_0)},x_0).
\end{equation}
By Lemma~\ref{SIGRA}, we write~$t_k:=\tau(x_k)$,
\[
r_k:=|x_k-x_0|,
\qquad{\mbox{and}}\qquad
 e_k:=\frac{x_k-x_0}{r_k}.
\]
Up to a subsequence, we may suppose that~$e_k\to e\in\mathbb S^{n-1}$ and that one of the following three alternatives holds for every~$k$: either~$t_k=t_0$, or~$t_k>t_0$, or~$t_k<t_0$.

Since~$\tau$ is locally Lipschitz (recall again Lemma~\ref{SIGRA}), we have that~$t_k\to t_0$. Also, since~$(x_0,t_0)\in\cS'$, we have that~$t_0\notin\cD$, thus~$E^{(t_k)}\to E^{(t_0)}$ locally in~$L^1(\R^n)$. 

Therefore, by Lemma~\ref{lem:variable-center-blowup}, up to a subsequence,
\[
C_k:=\frac{E^{(t_k)}-x_k}{r_k}\to C \qquad\text{locally in} \ L^1(\R^n),
\]
where~$C$ is an~$s$-minimal cone and~$0\in{\rm Sing}(\partial C)$.

Similarly,  up to a further subsequence we also have
\[
D_k:=\frac{E^{(t_0)}-x_0}{r_k}\to D
\qquad\text{locally in }L^1(\R^n),
\]
where~$D$ is an~$s$-minimal cone. Moreover,~$0\in{\rm Sing}(\partial D)$, otherwise the improvement of flatness used in the proof of Lemma~\ref{lem:variable-center-blowup} would give that~$x_0$ is a regular point of~$\partial E^{(t_0)}$.

In order to get a contradiction, the following observation will be used below. Let~$K$ be an~$s$-minimal cone in~$\R^n$, singular at the origin, and suppose that 
\begin{equation}\label{INCLUTI}
{\mbox{either~$K-e\subseteq K$ or~$K+e\subseteq K$.}}\end{equation} Then, by Lemma~\ref{lem:monotone-cones}, applied with~$e$ or with~$-e$, the cone~$K$ is invariant in the direction~$e$ (since the origin is singular, it cannot be a halfspace). Hence, writing~$K=\R e\times K'$ with~$K'\subset e^\perp\simeq\R^{n-1}$, the set~$K'$ is an~$s$-minimal cone in~$\R^{n-1}$ (see e.g. the dimensional reduction in~\cite{MR2675483})
that is not a halfspace. Since~$n-1=n_s^*$, this contradicts the definition of~$n_s^*$ in~\eqref{DIMECR}.

It remains to show that one of the inclusions in~\eqref{INCLUTI} necessarily occurs. To this end, we distinguish the three alternatives for~$t_k$.

If~$t_k=t_0$, then
\[
C_k=D_k-e_k.
\]
Passing to the limit, we obtain~$C=D-e$. In particular,~$D-e\subseteq C$. Lemma~\ref{lem:homogeneous-inclusinos}-(2) gives~$D\subseteq C=D-e$, and therefore~$D+e\subseteq D$; impossible by the observation after~\eqref{INCLUTI}.

Assume next that~$t_k>t_0$. By~\eqref{nested-maximal-branch}, we have~$E^{(t_0)}\subseteq E^{(t_k)}$. Hence,
\[
D_k-e_k\subseteq C_k.
\]
Letting~$k\to+\infty$, we find that~$D-e\subseteq C$. Lemma~\ref{lem:homogeneous-inclusinos}-(2) then yields~$D\subseteq C$, and since~$0\in\partial D\cap\partial C$, the strict maximum principle in~\cite{2023arXiv230801697D} gives $D=C$. Consequently,~$D-e\subseteq D$, which is again impossible by the observation above. 

The case $t_k < t_0$ is analogous, and so all three alternatives have led to a contradiction. Hence $\cS^*=\varnothing$, as desired.
\end{proof}

In particular, since we know 
\[
\pi_t(\cS) = \cI\cup \cD \cup \pi_t(\cS'\setminus \cS^*) \cup \pi_t(\cS^*), 
\]
combining Theorem~\ref{gene-uni}, Lemma~\ref{lem:Dcountable}, \eqref{eq:pitS's*}, and Proposition~\ref{prop:critical-generic-regularity}, we have just shown that the set of times $t$ for which there is singular set in dimension $n = n_s^*+1$ is countable.  We recall that the dimensional setting~$n=n_s^*+1$ also played a role
in the Bernstein Theorem for nonlocal minimal surfaces obtained in~\cite{MR3680376}.

In higher dimensions, we expect to reduce the size of the singular set generically as well.  

Specifically, setting
\[
m_s:=n-n_s^*-1,
\]
we have the following result:

\begin{lemma}
\label{lem:xi_j_halfspace}
Let $m_s:=n-n_s^*-1$ and $C$ be an $s$-minimal cone in $\R^n$.  Suppose that there are $m_s+1$ linearly independent vectors $\xi_1,\dots,\xi_{m_s+1} \in \R^n$ such that
\[
C-\xi_j\subseteq C \qquad\text{for every }j=1,\dots,m_s+1.
\]
Then $C$ is a halfspace.
\end{lemma}

\begin{proof}
 By Lemma~\ref{lem:homogeneous-inclusinos}-(1), we have that~$C-t\xi_j\subseteq C$ for all~$t>0$ and all~$j=1,\dots,m_s+1$, and hence, by Lemma~\ref{lem:monotone-cones} we know that~$C$ is invariant in each of the directions~$\xi_j$ (or $C$ is a halfspace directly).

Let~$V:={\rm span}\{\xi_1,\dots,\xi_{m_s+1}\}$ with $\dim V=m_s+1$. The invariance just proved gives~$C=V\times C'$, with~$C'\subset V^\perp\simeq\R^{n-m_s-1}=\R^{n_s^*}$ where~$C'$ is an~$s$-minimal cone. By the definition of~$n_s^*$ in~\eqref{DIMECR}, we have that~$C'$ is a halfspace, and so is~$C$.
\end{proof}

Let us also recall the following abstract GMT results from \cite{MR4179834}:
\begin{lemma}[\protect{\cite[Proposition 7.3]{MR4179834}}]
\label{LemReifenberg}
Suppose that $E\subset\R^n$ and $f:E\to\R$ satisfy the following:

For each $\eps>0$ and $x\in E$, there exists $\rho=\rho(x,\eps)>0$ such that for each $r\in(0,\rho)$, we can find an $m$-dimensional subspace $\Pi_{x,r}$, passing through $x$ and satisfying
$$
E\cap B_r(x)\cap \{y:~f(x)-\rho<f(y)<f(x)+\rho\}\subset\{y:~ \mathrm{dist}(y,\Pi_{x,r})\le \eps r\}.
$$

Then the Hausdorff dimension of $E$ can be  bounded from above as 
$$
\dim_{\mathcal H} (E)\le m.
$$
\end{lemma}

From which it follows: 
 
\begin{proposition}
\label{prop:dimred}
Assume that $n\ge n_s^*+2$.  Then
\[
\dim_{\mathcal H}\big(\pi_x(\cS)\big)\le n-n_s^*-1.
\]
\end{proposition}

\begin{proof}
Let~$m_s:=n-n_s^*-1$. We first prove that
\begin{equation}\label{DIMPISTAR}
\dim_{\mathcal H}\big(\pi_x(\cS^*)\big)\le m_s.
\end{equation}
For this, we check the hypothesis of~Lemma~\ref{LemReifenberg} for the set~$\pi_x(\cS^*)$ and the function~$f$ in \eqref{eq:fdef}.

Suppose, by contradiction, that this hypothesis fails at some point~$x_0\in\pi_x(\cS^*)$, and write~$t_0:=\tau(x_0)$. Then, for some~$\eps>0$, we can find~$r_k\searrow0$ and points
\[
x_k^{(1)},\dots,x_k^{(m_s+1)}\in \pi_x(\cS^*)\cap B_{r_k}(x_0)
\]
such that
\begin{equation}\label{FYDENS}
\big|f(x_k^{(j)})-f(x_0)\big|<\frac1k
\qquad{\mbox{for every }}j=1,\dots,m_s+1,
\end{equation}
and
\begin{equation}\label{FYPLANE}
\max_{j=1,\dots,m_s+1}{\rm dist}\big(x_k^{(j)},x_0+\Pi\big)\ge\eps r_k
\end{equation}
for every~$m_s$-dimensional vector subspace~$\Pi\subset\R^n$.

In view of Lemma~\ref{SIGRA},
we let~$t_k^{(j)}:=\tau(x_k^{(j)})$ and
\[
y_k^{(j)}:=\frac{x_k^{(j)}-x_0}{r_k}.
\]
By~\eqref{FYPLANE}, $y_k^{(j)}\in \overline B_1\setminus B_\eps$ and, up to a subsequence,
\[
y_k^{(j)}\to y^{(j)}\in\overline B_1\setminus B_\eps
\qquad{\mbox{for every }}j=1,\dots,m_s+1.
\]
with $y^{(1)},\dots,y^{(m_s+1)}$  linearly independent.

Since~$\tau$ is locally Lipschitz (recall again Lemma~\ref{SIGRA}), we have that~$t_k^{(j)}\to t_0$ for all~$j$. Also, since~$(x_0,t_0)\in\cS^*$, we know that~$t_0\notin\cD$, and therefore~$E^{(t_k^{(j)})}\to E^{(t_0)}$ locally in~$L^1(\R^n)$. 

From~\eqref{FYDENS} we also have
\[
\omega_s(E^{(t_k^{(j)})},x_k^{(j)})\to \omega_s(E^{(t_0)},x_0).
\]
Using Lemma~\ref{lem:variable-center-blowup}, we may assume, up to a further subsequence, that
\[
D_k:=\frac{E^{(t_0)}-x_0}{r_k}\to D
\qquad\text{and}\qquad
C_k^{(j)}:=\frac{E^{(t_k^{(j)})}-x_k^{(j)}}{r_k}\to C^{(j)}
\]
locally in~$L^1(\R^n)$, for every~$j=1,\dots,m_s+1$, where~$D$ and $C^{(j)}$'s are~$s$-minimal cones, and~$0\in {\rm Sing}(\partial D)\cap {\rm Sing}(\partial C^{(j)})$ (otherwise, improvement of flatness would yield a contradiction).

We now claim that, for each~$j=1,\dots,m_s+1$, one of the two inclusions
\begin{equation}\label{INCLUSIONSPLUSMINUS}
D-y^{(j)}\subseteq D
\qquad\text{or}\qquad
D+y^{(j)}\subseteq D
\end{equation}
holds true. To prove this, we fix~$j$ and, taking a further subsequence, we distinguish the three alternatives~$t_k^{(j)}=t_0$, $t_k^{(j)}>t_0$, and~$t_k^{(j)}<t_0$.

If~$t_k^{(j)}=t_0$, then~$C_k^{(j)}=D_k-y_k^{(j)}$, and consequently~$C^{(j)}=D-y^{(j)}$.  Thus, Lemma~\ref{lem:homogeneous-inclusinos}-(2) gives~$D\subseteq C^{(j)}=D-y^{(j)}$. Thus~$D+y^{(j)}\subseteq D$.

Assume next that~$t_k^{(j)}>t_0$. By~\eqref{nested-maximal-branch}, we have~$E^{(t_0)}\subseteq E^{(t_k^{(j)})}$, and therefore~$D_k-y_k^{(j)}\subseteq C_k^{(j)}$. Letting~$k\to+\infty$, we get~$D-y^{(j)}\subseteq C^{(j)}$ and Lemma~\ref{lem:homogeneous-inclusinos}-(2) implies~$D\subseteq C^{(j)}$. Since~$0\in\partial D\cap\partial C^{(j)}$, the strict maximum principle
in~\cite{2023arXiv230801697D} gives that~$D=C^{(j)}$, and consequently~$D-y^{(j)}\subseteq D$. The case $t_k^{(j)} <t_0$ is analogous. 

This proves~\eqref{INCLUSIONSPLUSMINUS}. Therefore, we find~$m_s+1$ linearly independent vectors~$\xi_1,\dots,\xi_{m_s+1}$ with
\[
D-\xi_j\subseteq D
\qquad\text{for every }j=1,\dots,m_s+1.
\]
Lemma~\ref{lem:xi_j_halfspace} then implies that~$D$ is a halfspace, in contradiction with~$0\in {\rm Sing}(\partial D)$.  

 It remains to pass from~$\cS^*$ to~$\cS$. We already know that~$\cI$,
 $\cD$, and~$\pi_t(\cS'\setminus\cS^*)$ are countable. Hence,
\[
\pi_x(\cS)\subseteq \pi_x(\cS^*)
\cup\bigcup_{t\in \cI\cup\cD\cup\pi_t(\cS'\setminus\cS^*)}{\rm Sing}(\partial E^{(t)}).
\]
By a standard dimension-reduction estimate for singular sets of nonlocal minimal surfaces, based on~\cite[Theorems~10.3 and~10.4]{MR2675483} and on the definition of~$n_s^*$, each set~${\rm Sing}(\partial E^{(t)})$ in the union above has Hausdorff dimension at most~$m_s$. 

Since the union is countable, this and~\eqref{DIMPISTAR} give
\[
\dim_{\mathcal H}\big(\pi_x(\cS)\big)\le m_s=n-n_s^*-1,
\]
as desired.
\end{proof}

We now prove the following refined version of the main results, Theorems~\ref{thm:main1} and~\ref{thm:main2}: 
\begin{theorem}\label{GENER:REGULA0}
Let~$G^{(0)}\subset\R^n\setminus B_1$, and let  $\{G^{(t)}\}_{t\in[0,1]}$ be a strictly increasing family of sets contained in~$\R^n\setminus B_1$ given by \eqref{def-special-family}. 

Let $E^{(t)}$ be an $s$-minimal set in $B_1$ with exterior datum $G^{(t)}$. Then,
if~$n=n_s^*+1$, there exists a countable subset~${\mathcal{J}}$ of~$[0,1]$ such that
\begin{equation}\label{SING:10} {\rm Sing}(\partial E^{(t)})=\varnothing \qquad{\mbox{ for all }}t\in[0,1]\setminus{\mathcal{J}}.\end{equation}

Also, if~$n\ge n_s^*+2$,
\begin{equation}\label{SING:20} \dim_{\mathcal H}({\rm Sing}(\partial E^{(t)}))\le n-n_s^*-2\qquad{\mbox{ for almost all }}t\in[0,1].\end{equation}
\end{theorem}

This theorem will follow as a consequence of the previous results together with the following abstract observation: 
\begin{lemma}[\protect{\cite[Corollary 7.8]{MR4179834}}]
\label{LemGenericReduction}
Suppose that for some $\sigma>0$ and $\gamma\in[\sigma,n]$, the set $E\subset\R^n\times[0,1)$ satisfies
\begin{enumerate}
\item{$\dim_{\mathcal{H}}(\pi_x(E))\le\gamma$; and }
\item{For each $(x_0,t_0)\in E$ and $\eps>0$,  there exists $\rho=\rho(x_0,t_0,\eps)>0$ such that 
$$
E\cap\{(x,t)\in B_{\rho}(x_0)\times[0,1):~ t-t_0>|x-x_0|^{\sigma-\eps}\}=\varnothing.
$$}
\end{enumerate}
Then, we have
$$
\dim_{\mathcal H}(E\cap\pi_t^{-1}(t))\le\gamma-\sigma~ \text{ for almost every } t\in[0,1).
$$
\end{lemma} 

With this, we can argue in the following way:
\begin{proof}[Proof of Theorem~\ref{GENER:REGULA0}]
The case $n = n_s^*+1$ now follows from Theorem~\ref{gene-uni}, Lemma~\ref{lem:Dcountable}, \eqref{eq:pitS's*}, and Proposition~\ref{prop:critical-generic-regularity}, as already mentioned. 

The case $n \ge n_s^*+2$ is a consequence of Lemma~\ref{LemGenericReduction}, by relying on Proposition~\ref{prop:dimred} and Corollary~\ref{cor:space-time-separation}. 
\end{proof}

Finally, in order to obtain the desired result
regarding $L^1_{\rm loc}$ small perturbations, we observe the following: 

\begin{lemma}\label{lem:small-exterior-perturbation0}
If $\partial A_0$ has locally finite $(n-1)$-dimensional
upper Minkowski content in $B_R$, then
\[
\big|(G^{(\delta)}\setminus G^{(0)})\cap B_R\big|
\le
O_R(\delta)
\quad\text{as }\delta\searrow0.
\]
In particular,
\[
G^{(t)}\to G^{(0)}
\quad\text{in }L^1_{\rm loc}(\mathbb R^n)
\quad\text{as }t\searrow0.
\]
\end{lemma}
\begin{proof}
     We claim that, for every $R>1$ and
every $t\in[0,1]$,
\begin{equation}\label{small-ext-containment}
(G^{(t)}\setminus G^{(0)})\cap B_R
\subseteq
B_R\cap \partial_{(R+1)t}A_0 .
\end{equation}

Indeed, let
\[
x\in (G^{(t)}\setminus G^{(0)})\cap B_R \subset A_t\setminus A_0.
\]
By the definition of $A_t$, there exist $q\in[0,t]$,
$y\in A_0$, and $z\in B_q$ such that
\[
x=(1+q)y+z .
\]
Since $x\in B_R$, we have
\[
|y|\le \frac{|x|+|z|}{1+q}
\le \frac{R+q}{1+q}\le R.
\]
Hence,
\[
|x-y|\le q|y|+|z|
\le (R+1)q
\le (R+1)t .
\]
Since $y\in A_0$ and $x\notin A_0$, this shows ${\rm dist}(x,\partial A_0)\le  (R+1)t$, i.e.,~\eqref{small-ext-containment}. In particular, 
\[
\big|(G^{(\delta)}\setminus G^{(0)})\cap B_R\big|
\le
|B_R\cap \partial_{(R+1)\delta}A_0|\le C_R \delta,
\]
where the last inequality holds by the definition of finite upper Minkowski content (recall Definition~\ref{def:minkowski_content}). This is the desired result.  
\end{proof}

We finally have: 

 \begin{proof}[Proof of Theorems~\ref{thm:main1} and \ref{thm:main2}]
 These results now follow by Theorem~\ref{GENER:REGULA0} together with Lemma~\ref{lem:small-exterior-perturbation0}. 

We notice that the assumption that $G^{(t)}$ is strictly increasing in
Theorem~\ref{GENER:REGULA0}
is harmless: if for some $0\le \tau < t \le 1$ we had $G^{(t)} = G^{(\tau)}$, then $A_\tau = \R^n$ (all up to sets of measure zero), and the minimizer is saturated. Indeed, 
as in the proof of Lemma~\ref{lem:linear-retraction}, one has
\[
\lambda A_\tau+B_r\subseteq A_t.
\]

Assume now that $A_t = A_\tau$. Then, up to sets of measure zero, we have $\lambda A_\tau + B_r \subset A_\tau$. Since $B_1\subset A_\tau$, this forces $B_{\lambda+r} = \lambda B_1 + B_r \subset A_\tau$, and iteratively, $B_{R_k}\subset A_\tau$ where $R_0 = 1$ and $R_{k+1} = \lambda R_k+r$. Since $\lambda > 1$, this implies $A_\tau = \R^n$. 
\end{proof}

\section*{Acknowledgments}

The first author was supported by the Australian Research Council Future Fellowship FT230100333 ``New perspectives on nonlocal equations''. 

The second author was supported by the Swiss National Science Foundation (SNF grant agreement PZ00P2\_208930), the Swiss State Secretariat for Education, Research and Innovation (SERI) under contract number MB22.00034, and the AEI project PID2024-156429NB-I00 (Spain).

The third author was supported by the Australian Laureate Fellowship FL190100081 ``Minimal surfaces, free boundaries and partial differential equations''.

\vfill
\end{document}